\documentclass[11pt,reqno]{amsart}
\usepackage[margin=1in]{geometry}
\usepackage{amssymb}
\usepackage{amsmath}
\usepackage{mathtools}
\usepackage{latexsym}
\usepackage{appendix}
\usepackage{mathrsfs}
\usepackage{amsthm,stmaryrd}
\usepackage{graphicx}
\usepackage{caption}
\usepackage{dsfont}
\usepackage{hyperref}
\hypersetup{hidelinks}
\usepackage{enumerate}
\usepackage{wasysym}
\usepackage{cancel}
\usepackage{color} 
\usepackage{comment}
\usepackage{xcolor}
\usepackage{bigints}
\usepackage{url}
\usepackage{bm}
\usepackage[utf8]{inputenc}
\usepackage[T1]{fontenc} 
\usepackage{adjustbox}
\usepackage{graphicx}
\usepackage{graphicx,tipa}

\numberwithin{equation}{section}
\theoremstyle{plain}
\newtheorem{thm}{Theorem}[section]  
\newtheorem{prop}[thm]{Proposition} 
\newtheorem{lem}[thm]{Lemma}
\newtheorem{cor}[thm]{Corollary}

\theoremstyle{definition}
\newtheorem{remark}[thm]{Remark}

\newtheorem{notation}[thm]{Notation}

\newcommand{\lap}{\Delta}

\newcommand{\eps}{\epsilon}
\newcommand{\df}{\coloneqq}
\newcommand{\eqn}[1]{\begin{equation*} #1 \end{equation*}}
\newcommand{\eqnum}[1]{\begin{equation} #1 \end{equation}}
\newcommand{\eqtext}[1]{\qquad \text{#1}}
\newcommand{\noi}{\noindent}
\newcommand{\parnoi}{\par \noi}
\newcommand{\pa}{\partial}
\newcommand{\RR}{\mathbb{R}}
\newcommand{\NN}{\mathbb{N}}
\newcommand{\BB}{\mathbb{B}}
\newcommand{\lm}{\lambda}
\newcommand{\Lm}{\Lambda}
\newcommand{\sg}{\sigma}
\newcommand{\om}{\omega}
\newcommand{\Om}{\Omega}
\newcommand{\smooth}[1]{\mathcal{C}^\infty(#1)}
\newcommand{\inprod}[3]{\int_{#1}\langle {#2},{#3}\rangle\,\dv_{\hspace{-0.25em}#1}}
\newcommand{\ip}[3]{\langle {#2},{#3}\rangle_{#1}}
\newcommand{\bigip}[2]{\big\langle #1\,,\,#2 \big\rangle}
\newcommand{\Snm}{{\mathbb{S}^{n-1}}}
\newcommand{\inn}[3]{\langle\hspace{-3pt}\langle {#2},{#3}\rangle\hspace{-3pt}\rangle_{#1}}

\DeclareMathOperator{\dv}{d\hspace{-0.1em}V}
\DeclareMathOperator{\dr}{dr}
\DeclareMathOperator{\dt}{dt}
\newcommand{\dvm}{\dv_{\hspace{-0.25em}M}}
\newcommand{\dvpam}{\dv_{\hspace{-0.25em}\pa M}}
\newcommand{\dvs}{\dv_{\hspace{-0.15em}\Snm}}
\newcommand{\karsg}{\bm{\sg}}
\newcommand{\karsgp}{\karsg^{(p)}}
\newcommand{\karsgpk}{\karsgp_k}
\newcommand{\karsgpkk}{\karsgp_{(k)}}
\newcommand{\karsgpkkp}{\karsgp_{(k+1)}}
\newcommand{\lmp}{\lm^{(p)}}
\newcommand{\lmpk}{\lmp_k}
\newcommand{\lmpkk}{\lmp_{(k)}}
\newcommand{\lmpkkp}{\lmp_{(k+1)}}
\newcommand{\lmpmm}{\lmp_{(m)}}
\newcommand{\p}{^{(p)}}
\newcommand{\pp}{^{(p+1)}}

\newcommand{\pam}{{\pa M}}
\newcommand{\half}{\frac{1}{2}}
\DeclareMathOperator{\sect}{Sect}
\DeclareMathOperator{\ric}{Ric}
\DeclareMathOperator{\wz}{W}
\DeclareMathAlphabet{\mathpzc}{OT1}{pzc}{m}{it}
\newcommand{\normal}{{\mathpzc{n}}}
\newcommand{\ray}{\mathcal{R}}
\newcommand{\bn}{{\BB^n}}
\newcommand{\phii}{\varphi}

\makeatletter
\def\subsection{\@startsection{subsection}{2}
  \z@{.5\linespacing\@plus.7\linespacing}{.1\linespacing}
  {\normalfont\bfseries}}
\makeatother
\allowdisplaybreaks

\begin{document}

\title[]{Eigenvalue bounds for the Steklov problem on differential forms in warped product manifolds}

\author{Tirumala Chakradhar}

\address{University of Bristol, School of Mathematics, Fry Building, Woodland Road, Bristol BS8 1UG, U.K.}

\email{tirumala.chakradhar@bristol.ac.uk}

\begin{abstract}
We consider the Steklov problem on differential $p$\nobreakdash-forms defined by M. Karpukhin and present geometric eigenvalue bounds in the setting of warped product manifolds in various scenarios. In particular, we obtain Escobar type lower bounds for warped product manifolds with non-negative Ricci curvature and strictly convex boundary, and certain sharp bounds for hypersurfaces of revolution, among others. We compare and contrast the behaviour with known results in the case of functions (i.e., $0$\nobreakdash-forms), highlighting the influence of the underlying topology on the spectrum for $p$\nobreakdash-forms in general.

\bigskip

\noindent\textsc{R\'esum\'e.}
Nous consid\'erons le probl\`eme de Steklov sur les formes diff\'erentielles de degr\'e~$p$, tel que d\'efini par M. Karpukhin, et pr\'esentons des bornes g\'eom\'etriques des valeurs propres dans le cadre des vari\'et\'es en produit d\'eform\'e dans divers sc\'enarios. En particulier, nous obtenons des bornes inf\'erieures de type Escobar pour les vari\'et\'es en produit d\'eform\'e ayant une courbure de Ricci non n\'egative et une fronti\`ere strictement convexe, ainsi que certaines bornes serr\'ees pour les hypersurfaces de r\'evolution, entre autres. Nous comparons et contrastons le comportement avec les r\'esultats connus dans le cas des fonctions (c'est-\`a-dire des $0$\nobreakdash-formes), en mettant en \'evidence l'influence de la topologie sous-jacente sur le spectre des $p$\nobreakdash-formes en g\'en\'eral.
\end{abstract}


\maketitle
%------------------------------------------
\section{Introduction} \label{sec: intro}
\par
Let $(M^n,g)$ be a compact, connected, smooth $n$\nobreakdash-dimensional Riemannian manifold with boundary, and $\lap\df~-\text{div}(\text{grad})$ be the positive Laplacian on $M$. We have the Steklov eigenvalue problem:
\eqn{ \label{steklov for func}
\begin{cases}
\lap u = 0 \qquad &\text{in } M,\\
\pa_\normal u=\sg u  \qquad &\text{on } \pa M,
\end{cases}
}
where $u\in \smooth{M}$ and $\sg$ is the spectral parameter. Here, $\pa_\normal $ denotes the partial derivative with respect to the unit outward normal vector field $\normal$ along the boundary. The Steklov problem is known to have a discrete spectrum $0=\sg_0< \sg_1\leq \sg_2\leq\dots\nearrow +\infty$, with finite multiplicities. It is often formulated in terms of the Dirichlet-to-Neumann (DtN) map \hbox{$\mathcal{D}:\smooth{\pa M}\to \smooth{\pa M}$} given by $u\mapsto \pa_\normal \widetilde u$, where $\widetilde u$ is the unique harmonic extension of $u$ to $M$. It is an elliptic, self-adjoint pesudo-differential operator whose spectrum coincides with the Steklov spectrum. Introduced by Vladimir Steklov~\cite{ste1902} in 1902, the Steklov problem arises in various interesting contexts such as the sloshing problem for liquids, Electrical Impedance Tomography, as well as the study of minimal surfaces. See~\cite{steklov_survey_one,steklov_survey_two} for a detailed survey of the recent advances in understanding the spectral geometry of the Steklov problem.
\par
In this article, we focus on geometric upper and lower bounds for the eigenvalues of the Steklov problem on differential forms. There are various versions of Dirichlet-to-Neumann operators on $p$\nobreakdash-forms in the literature~\cite{JL05,BS08,RS12,SS13,Kar19}, with the ones defined by Raulot\nobreakdash-Savo~\cite{RS12} and Karpukhin~\cite{Kar19} among those that drew attention recently. See \cite{RS14,GKLP22,Xio24,Mic19,YY17,Kwo16} for some interesting geometric eigenvalue estimates for these operators.
\par
The Hodge Laplacian generalises the notion of Laplacian on smooth functions to the framework of differential $p$\nobreakdash-forms on the manifold. The spectral data associated with the Hodge Laplacian, such as that of the absolute eigenvalue problem, the relative eigenvalue problem, etc., capture important geometric and topological properties of the underlying manifold. It turns out that a number of characteristic properties of various eigenvalue problems for functions (i.e., $0$-forms) do not generalise for arbitrary $p\in\{0,1,\dots,n\}$, $n$ being the dimension of the manifold. They show significantly different behavior, where the role of the topology of the manifold is much more eminent in addition to the differential structure. See~\cite{CC90} for a nice demonstration of this phenomenon. While the classical eigenvalue problems involving differential forms are well-studied, the Steklov problem on forms is relatively less explored, with many significant questions waiting to be investigated.
\par
In the case of functions, the recent articles~\cite{CGG19,CV21,Xio21,Xio22,Tsc24,BCG25,BC24} study various eigenvalue bounds for the Steklov problem among warped product manifolds, which are an interesting class of manifolds that are widely used in the literature to gain valuable insights into the behaviour of the spectrum, owing to certain simplifications in computing the eigenvalues and eigenfunctions. We generalise some of these estimates to the setting of the Steklov problem on $p$\nobreakdash-forms defined by Karpukhin, adapting their proofs. As we shall see, a striking feature in some of our results is the contrasting behaviour of eigenvalue bounds for small and large values of $p$, and, for even dimensional manifolds, an isospectrality with the Euclidean ball when $p=\frac{n-2}{2}$. We now motivate and describe our main results. See Section~\ref{sec: prelim} on preliminaries for the required definitions and the notation.
\par
We assume that $(M^n,g)$ is a compact, connected, orientable, smooth Riemannian manifold with boundary, and that $p\in\{0,1,\dots,n-2\}$, in addition to any other restrictions on $p$ stated in various contexts. Let $\karsgp_k$ (resp. $\karsgp_{(k)}$) denote the $k^\text{th}$ \emph{nonzero} eigenvalue of Karpukhin's DtN operator $\Lambda\p$ on coclosed $p$\nobreakdash-forms on $\pam$, counted with (resp. without) multiplicity. For $p=0$, i.e., for functions, $0$ is always an eigenvalue of $\Lambda\p$, so $\karsgp_1=\karsgp_{(1)}$ denotes the second smallest eigenvalue of $\Lambda\p$. In contrast, for $p\in \{1,\dots,n-2\}$, the manifolds that we study in this article (each connected component of whose boundary is the unit round sphere $\Snm$, up to scaling) have $\karsgp_1=\karsgp_{(1)}$ as the smallest eigenvalue as the kernel of $\Lambda\p$ turns out to be trivial for topological reasons.
\par
Escobar~\cite{esc99} conjectured that for a manifold $(M,g)$ with the Ricci curvature, $\ric_g\geq 0$ and strictly convex boundary (i.e., with a positive lower bound, say $\kappa$, on the principal curvature of $\pam$), the first nonzero Steklov eigenvalue for functions satisfies $$\karsg^{(0)}_1\geq \kappa,$$ with the equality holding if and only if $M$ is the Euclidean ball of radius $\kappa^{-1}$. While the conjecture remains open, some partial results are known in the literature. See~\cite{Xio22} for a brief survey, and~\cite{XX23} for the proof of a weaker version of the Escobar conjecture by replacing $\ric_g\geq 0$ with $\sect_g\geq0$ (where $\sect_g$ denotes the sectional curvature). Xiong~\cite[Theorem~2]{Xio22} proved certain Escobar type lower bounds for the $m^\text{th}$ nonzero Steklov eigenvalue for functions, that for warped product manifolds (see Section~\ref{subsec: warped prod mfds}) satisfying Escobar's hypotheses of $\ric\geq 0$ and the principal curvature of the boundary, $\kappa>0$, the following lower bounds are satisfied. $$\karsg^{(0)}_{(m)}\geq m\kappa,\qquad \forall \, m\in\NN,$$ with the equality holding only for the Euclidean ball of radius $\kappa^{-1}$.
\begin{remark}
It follows from the results of Ichida~\cite{ich81} and Kasue~\cite{kas83} that if a compact, connected manifold has $\ric\geq 0$ and a strictly convex boundary, then the boundary is connected. Also, note that the principal curvature is constant on each connected component of the boundary for warped product manifolds.
\end{remark}
\noi
We generalise the above result of Xiong to the Steklov problem on $p$\nobreakdash-forms in the following theorem.
\begin{thm}\label{thm: Xiong-type lower bound for p-forms}
Let $(M^n,g)$ be a warped product manifold as in Section~\ref{subsec: warped prod mfds} ($n\geq 2$), such that $\ric_g\geq 0$ and it has a strictly convex boundary. If $\kappa\in \RR_{>0}$ is the principal curvature of $\pam$, then, for $p\leq \frac{n-1}{2}$, the $m^\text{th}$ nonzero Steklov eigenvalue for $p$\nobreakdash-forms has the lower bound
\eqnum{\label{Xiong-type lower bound for p-forms}
\karsgp_{(m)}(M)\geq (m+p)\kappa, \qquad \forall \, m\in\NN,
}
with the equality holding for a given $m\in \NN$ if and only if $M$ is the Euclidean ball of radius $\kappa^{-1}$.
\end{thm}

\begin{remark}
Analogous to~\cite[Theorem~4]{Xio22}, we may instead use the hypothesis of $\ric_g(M)\leq 0$ and strictly convex boundary to get a sharp upper bound for $\karsgp_{(m)}(M)$ given by the same expression on the RHS of inequality~\eqref{Xiong-type lower bound for p-forms}, provided that $\pam$ is connected as this is not implied by the hypothesis, unlike in Theorem~\ref{thm: Xiong-type lower bound for p-forms}.
\end{remark}
\noi
Xiong~\cite{Xio24} extended the above mentioned weaker version of the Escobar conjecture for functions due to Xia and Xiong~\cite{XX23} to the setting of $p$\nobreakdash-forms, that for each $p\in \{0,1,\dots,n-2\}$, a manifold with strictly convex boundary, $\sect_g\geq0$ and the Weitzenb\"ock curvature $\wz\pp_g\geq0$, has the following lower bound for its first nonzero Steklov eigenvalue.
$$\karsgp_1\geq (1+p)\kappa,$$
with the equality holding for the Euclidean ball of radius $\kappa^{-1}$. Here, $\kappa\in \RR_{>0}$ denotes the infimum of the principal curvature of $\pam$. Note that, when restricted to the class of warped product manifolds, our sharp lower bound in Theorem~\ref{thm: Xiong-type lower bound for p-forms} coincides with this, under a weaker hypothesis of nonnegative Ricci curvature. In addition, we also have rigidity for the equality case in Theorem~\ref{thm: Xiong-type lower bound for p-forms}, which is not known for the result in~\cite{Xio24}. Further, it would be interesting to see if we can replace our hypothesis on the Ricci curvature and strictly convex boundary with some assumptions on the Weitzenb\"ock curvature and the $p$\nobreakdash-curvatures (see, e.g.,~\cite{RS12}) to get meaningful bounds, preferably with the restrictions on $p$ in Theorem~\ref{thm: Xiong-type lower bound for p-forms} removed.
\par
With similar hypotheses on Ricci curvature and convexity of the boundary, Xiong~\cite{Xio21} also proved optimal upper and lower bounds for the eigenvalue gaps and the ratios of successive eigenvalues for functions. In a more recent work, Brisson, Colbois and Gittins~\cite[Theorems~1.1, 1.2, 1.5~\&~1.7]{BCG25} considered warped product manifolds with connected boundary with no further hypothesis on the curvature or the convexity of the boundary and obtained upper bounds for the eigenvalue gaps and the ratios in the case of functions, in terms of the gaps and ratios for the corresponding Laplace eigenvalues of the Euclidean unit sphere $\Snm$. We show that analogous results hold for $p$\nobreakdash-forms in general.
\begin{thm}\label{thm: bcg ratio upperbound for p-forms}
Let $(M^n,g)$ be a topological ball with revolution-type metric, i.e., a warped product manifold as in Section~\ref{subsec: warped prod mfds}, with connected boundary. Then,
\begin{enumerate}[(i)]
\item For $n\geq3$ and $p\in\{0,\dots,n-2\}$,
\eqn{
\frac{\karsgpkkp(M)}{\karsgpkk(M)}<\frac{\lmp_{(k+1)}(\Snm)}{\lmp_{(k)}(\Snm)}=\frac{(k+p+1)(n+k-p-1)}{(k+p)(n+k-p-2)},\qquad \forall \, k\in \NN,
}
and, if $p\leq \frac{n-3}{2}$, this upper bound is sharp. Here, $\{\lmp_{(k)}\}_{k\in \NN}$ denote the positive eigenvalues of the Hodge Laplacian acting on the space of coclosed forms $\text{cc}\Om^p(\Snm)$, counted without multiplicity.\vspace{1em}
\item When $n\geq2$ is even and $p=\frac{n-2}{2}$,
\eqn{
\karsgpkk(M)=\frac{\karsgpkk(\bn)}{h_0}=\frac{k+p}{h_0},
}
where $h_0\in\RR_{>0}$ is the value of the associated warping function on $\pam$. Here, $\bn$ denotes the Euclidean unit ball.
\end{enumerate}
\end{thm}

\begin{thm}\label{thm: bcg spectral gaps upperbound for p-forms}
Let $(M^n,g)$ be a topological ball with revolution-type metric. Then,
\begin{enumerate}[(i)]
    \item For $n\geq 4$ and $p\leq \frac{n-4}{2}$, if the associated warping function satisfies $h(r)\leq C$ on $r\in [0,R]$ for a constant $C>0$,
    \eqn{
    \karsgpkkp(M)-\karsgpkk(M)\leq \frac{RC^{n-2p-3}}{h_0^{n-2p-1}}\big(\lmp_{(k+1)}(\Snm)-\lmp_{(k)}(\Snm)\big).
    }
    \item For $n\geq 3$ odd and $p=\frac{n-3}{2}$,
    \eqn{
    \karsgpkkp(M)-\karsgpkk(M)\leq \frac{R}{h_0^{n-2p-1}}\big(\lmp_{(k+1)}(\Snm)-\lmp_{(k)}(\Snm)\big).
    }
\end{enumerate}
Further, it follows that a uniform lower bound on $h_0$ would give uniform upper bounds for the spectral gaps above, independent of the metric/warping function. As before, $h_0$ here denotes the value of the warping function on $\pam$.
\end{thm}
\par
We now consider hypersurfaces of revolution, described in Section~\ref{subsec: warped prod mfds}. The warping function associated with this class of warped product manifolds satisfies certain additional conditions (see, e.g., Section~\ref{sec: proofs of cgg thms for p-forms}) that can be utilised in the study of eigenvalues. Colbois, Girouard and Gittins~\cite{CGG19} proved various upper and lower bounds for the Steklov eigenvalues for functions, among the class of hypersurfaces of revolution. We present certain generalisations to the setting of $p$\nobreakdash-forms, with contrasting behaviour for different values of $p$.
\parnoi
The following theorem considers the scenario of a connected boundary $\Snm$ and states that the Steklov eigenvalues are bounded from below by those of the Euclidean unit ball, for $p$ small enough. More precisely,
\begin{thm}\label{thm: thm 1.8 of cgg lowerbound for one bdry compt for p-forms}
(cf.~\cite[Theorem~1.8]{CGG19}) Let $M^n\subset\RR^{n+1}$ be a hypersurface of revolution with boundary $\Snm\times\{0\}$.
\begin{enumerate}[(i)]
    \item For $n\geq 3$ and $p\leq \frac{n-3}{2}$,
    \eqn{\karsgpk(M)\geq\karsgpk(\bn),\qquad \forall\, k\in\NN,}
    with the equality holding if and only if $M$ is the Euclidean unit ball $\bn\times\{0\}$. \vspace{1em}
    \item For $n\geq2$ even and $p=\frac{n-2}{2}$, $M$ is isospectral to the Euclidean unit ball:
    \eqn{\karsgpk(M)=\karsgpk(\bn),\qquad \forall\, k\in\NN.}
\end{enumerate}
\end{thm}
When there are two boundary components each of which is isometric to $\Snm$, we prove, for sufficiently large meridian length, sharp lower bounds for the Steklov eigenvalues for $p$\nobreakdash-forms in terms of the eigenvalues of the Euclidean ball. For small meridian lengths, we have non-sharp lower bounds that depend on the meridian length, when $p$ is small enough. In addition to these results which generalise the previously known lower bounds in the case of functions, we further have, for small meridian lengths, upper bounds when $p$ is large enough. 
\begin{thm}\label{thm: thm 1.11 of cgg lower bound for two bdry compts for p-forms}
(cf.~\cite[Theorem~1.11]{CGG19}) Let $M^n\subset \RR^{n+1}$ be a hypersurface of revolution ($n\geq3$) with boundary $\Snm\times\{0,d\}\subset \RR^{n+1}$, $L$ be the intrinsic distance between the boundary components (meridian length).
\begin{enumerate}[(i)]
    \item If $L\geq2$, then for each $p\leq \frac{n-3}{2},$
    \eqn{
    \karsgpk(M)\geq \karsgpk(\bn\sqcup\bn),\qquad \forall \, k\in\NN,
    } and the inequality is sharp. \vspace{1em}
    \item When $0<L\leq 2$:\vspace{.5em}
    \begin{enumerate}[(a)]
        \item If $p\leq \frac{n-3}{2}$, then
        \eqn{
        \karsgpk(M)\geq\Big(1-\frac{L}{2}\Big)^{n-2p-1}\karsgpk(C_L),\qquad \forall \,k\in\NN.
        }
        \item If $p\geq\frac{n-1}{2}$, then
        \eqn{
        \karsgpk(M)\leq \Big(1-\frac{L}{2}\Big)^{n-2p-1}\karsgpk(C_L) ,\qquad \forall\,k\in\NN. 
        }
    \end{enumerate}Here, $C_L$ denotes the standard cylinder $\Snm\times [0,L]$.
\end{enumerate}
\end{thm}
\noi The proof of Theorem~\ref{thm: prop3.3 of cgg misc bounds for two bdry compts for p-forms} in Section~\ref{sec: proofs of cgg thms for p-forms} also provides the (unordered) set of eigenvalues $\{\karsgpk(C_L)\}_{k\in \NN}$ in terms of the coclosed Hodge Laplace eigenvalues of $\Snm$ and the meridian length. 
\par We also have the following non-sharp upper bounds for the Steklov eigenvalues for $p$\nobreakdash-forms in terms of the eigenvalues of standard cylinders and the meridian length, for $p$ small enough. We further have new types of lower bounds when $p$ is large enough, which are not applicable for the case of functions. 
\begin{thm}\label{thm: prop3.3 of cgg misc bounds for two bdry compts for p-forms}
(cf.~\cite[Proposition~3.3]{CGG19}) Let $M^n$ be as in Theorem~\ref{thm: thm 1.11 of cgg lower bound for two bdry compts for p-forms}. For any $L>0$,
\begin{enumerate}[(i)]
    \item If $p\leq\frac{n-3}{2}$, then
    \eqn{
    \karsgpk(M)\leq \Big(1+\frac{L}{2}\Big)^{n-2p-1}\karsgpk(C_L),\qquad \forall\,k\in\NN.
    }
    \item If $p\geq\frac{n-1}{2}$, then
    \eqn{
    \karsgpk(M)\geq \Big(1+\frac{L}{2}\Big)^{n-2p-1}\karsgpk(C_L), \qquad \forall\,k\in\NN.
    }
\end{enumerate}
The RHS goes to $0$ as $L\to0$ in each case.
\end{thm}
\par
An interesting corollary to Theorems~\ref{thm: thm 1.11 of cgg lower bound for two bdry compts for p-forms}~and~\ref{thm: prop3.3 of cgg misc bounds for two bdry compts for p-forms} would be the isospectrality with the standard cylinder when $p=\frac{n-1}{2}$, for sufficiently small meridian length.
\begin{cor}
Let $M^n$ be as in Theorem~\ref{thm: thm 1.11 of cgg lower bound for two bdry compts for p-forms}. If $0<L<2$, then we have, for $p=\frac{n-1}{2}$,
$$\karsgpk(M)=\karsgpk(C_L),\qquad \forall \,k\in\NN.$$
\end{cor}
\par
The results for $p$\nobreakdash-forms in this article clearly depend on the value of $p$. For instance, the methods used to obtain lower bounds for a certain value of $p$ might instead give upper bounds for a different value of $p$, and vice versa. As we shall see in the later sections, this behaviour is attributed to the presence of terms such as $h^{n-2p-1}$ and $h^{n-2p-3}$ in the Rayleigh quotients on warped product manifolds, where $h$ is the warping function. It is natural to ask what happens for the cases of $p$ that our results do not cover. Of particular interest is the case $p=\frac{n-2}{2}$ where we get certain isospectrality in some cases (e.g., Theorem~\ref{thm: thm 1.8 of cgg lowerbound for one bdry compt for p-forms}), but the nature of the bounds remains unknown in the other cases where we cannot directly compare the Rayleigh quotients involved using the bounds on the warping functions. See, for example, Remark~\ref{rmk: why not upper bds for large p}. On the other hand, one can also investigate if similar eigenvalue bounds can be obtained for the Steklov problem of Raulot-Savo, for which the present methods are not directly applicable due to the presence of an additional term in the Rayleigh quotient involving the codifferential. Note that, in view of comparison~\eqref{karpukhin and raulot-savo nonzero eigval comparison}, some of the upper bounds obtained for $\karsgpk$ in this article would also serve as upper bounds for $\sg\p_k$, the $k^{th}$ positive eigenvalue of the Steklov problem of Raulot-Savo (described in Section~\ref{subsec: Steklov prob on p-forms intro}).

\par
This article is organised as follows. In Section~\ref{sec: prelim} on preliminaries, we introduce the Steklov eigenvalue problem on $p$\nobreakdash-forms and describe warped product manifolds, after fixing some notation. The remaining sections deal with the proofs of the main results. In Section~\ref{sec: proof of xiong-type bounds}, we show that the Steklov eigenforms on warped product manifolds with connected boundary can be characterised using certain separation of variables and associate with it an ODE involving the warping function. We then analyse the ODE using the hypotheses on the Ricci curvature and the boundary convexity to obtain the sharp lower bounds of Theorem~\ref{thm: Xiong-type lower bound for p-forms}. In Section~\ref{sec: proofs of bcg thms for p-forms}, we prove the upper bounds of Theorems~\ref{thm: bcg ratio upperbound for p-forms}~\&~\ref{thm: bcg spectral gaps upperbound for p-forms} for the spectral gaps and the ratios  by using the special structure of the Rayleigh quotients for the Steklov eigenforms for warped products and considering appropriate test forms in the variational characterisation. In Section~\ref{sec: proofs of cgg thms for p-forms}, we prove the geometric upper and lower bounds of Theorems~\ref{thm: thm 1.8 of cgg lowerbound for one bdry compt for p-forms}\nobreakdash-\ref{thm: prop3.3 of cgg misc bounds for two bdry compts for p-forms} by comparing the warping functions of hypersurfaces of revolution with those of the Euclidean ball and the standard cylinder, which directly influence the behaviour of the corresponding Rayleigh quotients. We emphasise the role of $p$ in determining whether these comparisons give rise to upper bounds or lower bounds for $\karsgpk$ in different scenarios.
%---------------------------------------------
\section{Preliminaries} \label{sec: prelim}
\subsection{Notation} \label{subsec: main notation in intro}
\par
Throughout this article, $(M^n,g)$ denotes a smooth $n$-dimensional Riemannian manifold with non-empty boundary $\pa M$. We assume that $M$ is orientable, compact and connected.

Let $\Om^p(M)$ denote the space of smooth differential $p$\nobreakdash-forms on $M$. We use the standard notation $\ip{}{.}{.}$ to denote the natural pointwise inner product of two differential forms given by the underlying Riemannian metric, and the notation $|\alpha|$ to mean $\ip{}{\alpha}{\alpha}^\half$. Let $\inn{}{.}{.}$ denote the $L^2$-inner product of square integrable forms defined by $\inn{L^2\Om^p(M)}{\alpha}{\beta}\df\inprod{M}{\alpha}{\beta}$, where $\dvm$ is the Riemannian volume form on $M$. For $p\in \{0,1,\dots,n\}$, let  $\star:\Om^p(M)\to \Om^{n-p}(M)$ denote the Hodge star operator defined by the condition that $\alpha\wedge \star \om=\ip{}{\alpha}{\om}\dvm$ for all $\alpha \in \Om^p(M)$.  For $p\in \{0,1,\dots,n-1\}$, let $d:\Om^p(M)\to\Om^{p+1}(M)$ denote the exterior derivative and for $p\in \{1,2,\dots,n\}$, let \hbox{$\delta:\Om^{p}(M)\to\Om^{p-1}(M)$} be the codifferential operator given by $\delta \om= (-1)^{np+n+1}\star d (\star \om)$. The Hodge-Laplacian $\lap ^{(p)}:\Om^p(M)\to \Om^p(M)$ is defined to be $\lap^{(p)}\df d\delta+\delta d$, where we use the convention $\delta d\om=0$ for $\om\in\Om^n(M)$ and $d\delta\om=0$ for $\om\in\Om^0(M)$. Note that $\lap^{(0)}$ corresponds to the usual Laplace-Beltrami operator on $\smooth{M}$ defined to be $\lap^{(0)}f=-\text{div}(\text{grad}f)$. We often drop the superscript ${(p)}$ in $\lap^{(p)}$ when clear from the context.

\par For $\om \in \Om^p(M)$, let $\textbf{t}\om$ be the tangential part of $\om$ on $\pa M$, i.e., $\textbf{t}\om(X_1,\dots,X_p)=\om(X_1^\parallel,\dots,X_p^\parallel)$ and $\textbf{n}\om$ be the normal part of $\om$ on $\pa M$, i.e., $\textbf{n}\om(X_1,\dots,X_p)=\om(X_1^\perp,\dots,X_p^\perp)$. Here, $X^\parallel$ and $X^\perp$ denote the components of a vector field $X$ on $M$, tangential and normal to $\pa M$, respectively. Closely related to $\textbf{t}\om$ and $\textbf{n}\om$ are the forms $i^*\om$ and $i_\normal \om$, which are defined to be the pull back of $\om$ under the inclusion map $i:\pa M\to M$, and the interior product of $\om$ with respect to the outer unit normal vector field $\normal$, respectively. We refer to~\cite[Section~1.2]{hodge_text} for more details on the basic properties of the operators defined above, and their expressions in terms of the covariant derivatives of $p$\nobreakdash-forms.

\par We also use various Sobolev spaces in the framework of $p$\nobreakdash-forms. It is somewhat technical to introduce them rigorously and we refer to~\cite[Section~1.3]{hodge_text} for a detailed description of the Sobolev spaces $W^{k,s}\Om^p(M)$ defined in terms of generalised covariant derivatives. We use the notation $L^s\Om^p(M)$ for the space $W^{0,s}\Om^p(M)$. For $s=2$, we have a natural Hilbert space, which we denote by $H^k\Om^p(M)$.

\subsection{Steklov problem on differential forms} \label{subsec: Steklov prob on p-forms intro}
\par
Analogous to the case of functions, one can consider a DtN operator that maps $\om\in \Om^p(\pam)$ to $i_\normal d\widetilde \om\in \Om^p(\pam)$, where $\widetilde \om$ is a harmonic extension. However, in contrast to the case of functions, the harmonic extension of a boundary $p$\nobreakdash-form to $M$ is not unique in general for $p\in \{1,2,\dots,n-1\}$. Further conditions are imposed to obtain a well-defined spectral problem, which can be done in various ways. The choice of the additional condition is often motivated by the Hodge decomposition and the following observation that is of relevance in using the weak formulations. For $\alpha, \beta\in \Om^p(M)$,
\begin{align*}
    &\inprod{M}{\lap \alpha}{\beta}=\inprod{M}{d\delta \alpha}{\beta}+\inprod{M}{\delta d \alpha}{\beta}\\
    &=\inprod{M}{d\alpha}{d\beta}+\inprod{M}{\delta \alpha}{\delta\beta}+\inprod{\pa M}{i^*\delta \alpha}{i_\normal \beta}-\inprod{\pa M}{i_\normal d\alpha}{i^*\beta}.
\end{align*}
This is easy to see using the Green's identity: For $\alpha \in \Om^p(M)$ and $\beta\in \Om^{p+1}(M)$,
\eqnum{\label{green's identity}
\inprod{M}{d\alpha}{\beta}-\inprod{M}{\alpha}{\delta \beta}=\inprod{\pa M}{i^*\alpha}{i_\normal \beta}.
}
\parnoi
\textbf{Hodge-Morrey-Friedrichs decomposition \textnormal{\cite[Section~2.4]{hodge_text}}:} For $p\in \{0,1,\dots,n\}$ the space $\Om^p(M)$ can be written as the $\inn{L^2\Om^p(M)}{}{}$-orthogonal decomposition of exact forms, coexact forms and harmonic fields, with appropriate boundary conditions as below:
\eqn{
\Om^p(M)=\overline{\mathcal{E}}^p(M)\oplus \overline{c\mathcal{E}}^p(M)\oplus \mathcal{H}^p(M),
}
where
\begin{align*}
\overline{\mathcal{E}}^p(M)&\df \{\om\in\Om^p(M):\exists\, \alpha\in \Om^{p-1}(M) \text{ such that }\om=d\alpha,i^*\alpha=0\}, \\
\overline{c\mathcal{E}}^p(M)&\df \{\om\in\Om^p(M):\exists\, \alpha\in \Om^{p+1}(M) \text{ such that }\om=\delta\alpha,i_\normal \alpha=0\},  \\
\mathcal{H}^p(M)&\df\{\om\in\Om^p(M):d\om=\delta\om=0\}.
\end{align*}

\noi We now state the Steklov problem on forms that we study in this article.
\parnoi
\textbf{Karpukhin's Steklov problem \textnormal{\cite{Kar19}}: }For $\om\in \Om^p(M)$,
\eqn{
\begin{cases}
\lap^{(p)} \om = 0 \qquad &\text{in } M\\
\delta \om =0 \qquad &\text{in } M\\
i_\normal d\om=\sg \,i^*\om \qquad &\text{on } \pa M,
\end{cases}
}
where $\sg\in\RR$ is the spectral parameter. Equivalently, the DtN operator is defined as follows.
\begin{align*}
\Lm^{(p)}:\,&\Om^p(\pa M)\to \Om^p(\pa M)\\
&\om \mapsto i_\normal d\widetilde \om,
\end{align*}
where $\widetilde \om$ is an extension of $\om$ to $M$ satisfying the conditions $\lap \widetilde \om=0$ and $\delta \widetilde \om=0$ in $M$. Here, the extension is unique up to an element of the space of normal harmonic fields \eqn{\mathcal{H}_D\Om^p(M)\df \{\om\in \Om^p(M): d\om=0,\,\delta\om=0,\, \textbf{t}\om=0\},}
and it can be shown that this operator is well-defined. See~\cite[Section~2.3]{Kar19}. In this article, we refer to the unique extension of $\om$ orthogonal to $\mathcal{H}_D \Om^p(M)$ to be (Karpukhin's) Steklov eigenform associated with $\om$.
\begin{remark}
It turns out that $\Lm^{(n-1)}\equiv 0$ on $\Om^{n-1}(\pa M)$. See~\cite[Remark~2.6]{Kar19}. Henceforth in this article, it is understood that $p\in\{0,1,\dots,n-2\}$, along with any other assumptions on $p$ being stated.
\end{remark} \vspace{-2ex}
\noi The operator $\Lm^{(p)}$ is zero on the space of exact forms \hbox{$\{\om\in \Om^p(\pa M):\om=d\alpha \text{ for }\alpha\in\Om^{p-1}(\pa M)\}$}. When restricted to the space of coclosed forms $\text{cc}\Om^p(\pam)\df\{\om\in \Om^p(\pa M):\delta \om=0\}$, it has a discrete spectrum $$0=\karsgp_{0,1}=\dots=\karsgp_{0,\dag}<\karsgp_1\leq \karsgp_2\leq\dots \nearrow +\infty$$ with finite multiplicities, where $\dag\in \NN\cup \{0\}$ denotes the dimension of the kernel of $\Gamma\p$ on $\text{cc}\Om^p(\pam)$. The variational characterisation is given by
\eqn{
\karsgp_k(M)=\min_{\begin{subarray}
{c}V \subset \text{cc}\Om^p(M)\\ \dim(V)=k+\dag \end{subarray}}\,\max_{0\neq \om \in V}
\frac{\int_M|d\om|^2\,\dvm}{\int_{\pa M} |i^*\om|^2\,\,\dvpam}, \qquad k\in \NN.
}The space $\text{cc}\Om^p(M)$ in the above expression may also be replaced by the Sobolev space $H^1\text{cc}\Om^p(M)$, which is often useful. The kernel of $\Lm^{(p)}|_{\text{cc}\Om^p(M)}$ is equal to the space of coclosed $p$\nobreakdash-forms on $\pa M$ that are restrictions of the tangential harmonic fields,
\eqn{\mathcal{H}_N\Om^p(M)\df \{\om\in \Om^p(M): d\om=0,\,\delta\om=0,\, \textbf{n}\om=0\},}
to the boundary. The dimension of $\text{ker}\,\Lm^{(p)}|_{\text{cc}\Om^p(M)}$ is given by $$\dim\text{ker}\,\Lm^{(p)}|_{\text{cc}\Om^p(M)}=\dim \text{im} \{i^*:H^p(M)\to H^p(\pam)\}.$$ \emph{Henceforth, $\Lm\p$ is assumed to mean its restriction to $\text{cc}\Om^p(\pam)$.}
\par The following version of the Steklov problem on $p$\nobreakdash-forms is also of interest in the study of geometric eigenvalue bounds. For this version, we may relax the assumption that $M$ is orientable.
\parnoi
\textbf{Raulot-Savo's Steklov problem \textnormal{\cite{RS12}}: } For $\om\in \Om^p(M)$,
\eqn{
\begin{cases}
\lap^{(p)} \om = 0 \qquad &\text{in } M\\
i_\normal \om =0 \qquad &\text{on } \pa M\\
i_\normal d\om=\sg \,i^*\om \qquad &\text{on } \pa M,
\end{cases}
}
where $\sg\in\RR$ is the spectral parameter. Equivalently, the DtN operator is defined as follows.
\begin{align*}
T^{(p)}:\,&\Om^p(\pa M)\to \Om^p(\pa M)\\
&\om \mapsto i_\normal d\widetilde \om,
\end{align*}
where $\widetilde \om$ is the unique extension of $\om$ to $M$ satisfying the conditions $\lap \widetilde \om=0$ in $M$ and $i_\normal \widetilde \om=0$ on $\pa M$. The pseudo-elliptic, self-adjoint operator $T^{(p)}$ has a discrete spectrum $$0=\sg^{(p)}_{0,1}=\dots=\sg^{(p)}_{0,\natural}<\sg^{(p)}_1\leq \sg^{(p)}_2\leq \dots\nearrow +\infty$$ with finite multiplicities, where $\natural\in\NN\cup\{0\}$ denotes the dimension of the kernel of $T\p$. The variational characterisation is given by
\eqn{
\sg^{(p)}_k(M)=\min_{\begin{subarray}
{c}V \subset \{\alpha\in\Om^p(M):i_\normal \alpha=0\}\\ \dim(V)=k+\natural\end{subarray}}\,\max_{\begin{subarray}{c} 0\neq \om \in V
\end{subarray}}
\frac{\int_M(|d\om|^2+|\delta \om|^2)\,\dvm}{\int_{\pa M} |i^*\om|^2\,\,\dvpam},\qquad k\in \NN.
}
It is interesting to note that the kernel of $T^{(p)}$ is isomorphic to the absolute de Rham cohomology of $M$, and thus has its dimension equal to the $p$-th absolute Betti number of $M$.
\begin{remark}
We have the following comparison~\cite[Theorem~2.5]{Kar19} for the $k^\text{th}$ positive eigenvalues of $T\p$ and $\Lambda\p$:
\eqnum{\label{karpukhin and raulot-savo nonzero eigval comparison}
\sg^{(p)}_k\leq \karsgpk,\qquad \forall\, k \in \NN.
}
\end{remark}

\subsection{Warped product manifolds}\label{subsec: warped prod mfds}
\par
We consider warped product manifolds of the form $[0,R]\times \mathbb{S}^{n-1}$ with the Riemannian metric $$g=dr^2\oplus h(r)^2 g_{\mathbb{S}^{n-1}},$$ for a smooth positive function $h(r)$ defined on $[0,R]$. With slight abuse of notation, this also includes manifolds with connected boundary by allowing $h(R)=0$, in which case we impose an additional hypothesis on $h$ to ensure the smoothness of the metric, given by Assumption~(A) below.
\parnoi
\emph{Assumption~(A): }$h\in \smooth{[0,R]}$, $h(r)>0$ for $r\in [0,R)$, $h(R)=0$, $h'(R)=-1$ and all even order derivatives of $h$ are zero at $r=R$.
\par
Observe, for instance, that the warping function $h(r)=R-r$ corresponds to $M$ being a Euclidean ball, while $h(r)\equiv 1$ corresponds to the standard cylinder. We refer the interested reader to~\cite[Section~4.2]{petersen_text} for a general discussion on warped product manifolds and their significance. For the warped product manifolds being considered, the boundary is totally umbilical, with the corresponding principal curvatures of $\{0\}\times \Snm$ and $\{R\}\times \Snm$ given by $-\frac{h'(0)}{h(0)}$ and $\frac{h'(R)}{h(R)}$, respectively. Similarly, one can calculate the sectional and Ricci curvatures of $M$ in terms of the warping function $h$ (see, for example,~\cite{Xio22}).
\par 
A manifold $M$ is said to have strictly convex boundary if all the principal curvatures are strictly positive at each point on $\pam$. This terminology is used in Theorem~\ref{thm: Xiong-type lower bound for p-forms}.
\begin{remark}
Recall that the dimension of $\ker \Lm^{(p)}$ is equal to $\dim \text{im} \{i^*:H^p(M)\to H^p(\pam)\}$. Since we have $\pam$ to be topologically $\Snm$ or $\Snm\sqcup\,\Snm$ (with trivial cohomology for \linebreak[4]\hbox{$p\in\{1,\dots,n-2\}$}), it follows that, for $p\neq0$, $\ker \Lm^{(p)}=0$ and the first eigenvalue is strictly positive, in contrast to the case of functions where the first Steklov eigenvalue is always $0$.
\end{remark} \vspace{-2ex}
Further, as we shall see, separation of variables in the case of warped products allows us to study the spectrum indexed without counting the multiplicities. 
\begin{notation} For $k\in \NN$, let $\karsgpkk$ denote the $k^{\text{th}}$ positive eigenvalue of $\Lm^{(p)}$ counted \emph{without} multiplicity.
\end{notation} \vspace{-2ex}
We now introduce hypersurfaces of revolution, which are an important family of warped products and the setting in which we study Steklov eigenvalue bounds in Theorems~\ref{thm: thm 1.8 of cgg lowerbound for one bdry compt for p-forms}\nobreakdash-\ref{thm: prop3.3 of cgg misc bounds for two bdry compts for p-forms}. Throughout this article, a manifold $M^n\subset \RR^{n+1}$ is said to be a hypersurface of revolution if it can be parametrised by a map,
\begin{gather*}
\psi:\Snm\times [0,L]\to\RR^{n+1}\\*
(x,r)\mapsto h(r)x+z(r)e_{n+1},
\end{gather*}
where $\Snm\subset\RR^n$; $h,z$ are smooth functions on $[0,L]$ satisfying $h(0)=1,\,h(r)>0$ for $r\in [0,L)$ and such that the curve $r\mapsto h(r)x+z(r)e_{n+1}$ is parametrised by its arclength. Unless stated otherwise, we refer to the induced metric on $M$ from $(\RR^{n+1},g_{Euc})$, which can be written as a warped product described above, with $h$ as the warping function. When the boundary is connected, we have, in addition, that $h(L)=0$ (as part of Assumption~(A), with $R$ replaced by $L$). See~\cite[Section~3.1]{CGG19} for details.
%--------------------------------------------

\section{Proof of Theorem\texorpdfstring{~\ref{thm: Xiong-type lower bound for p-forms}}{ 1.2}}\label{sec: proof of xiong-type bounds}
\par
Recall from Section~\ref{sec: intro} that under the hypotheses of Theorem~\ref{thm: Xiong-type lower bound for p-forms}, $\pam$ is connected. In this section, we use the change of variable,
\begin{gather*}
\Theta:[0,R]\to[0,R],\\
\Theta(r)\df R-r,
\end{gather*}
for the convenience of the reader to compare the proof with that for the case of functions from~\cite{Xio22}. With slight abuse of notation, all the functions being used throughout \textit{this section}, including the warping function $h$, are understood to mean the respective functions composed with $\Theta$. Assumption~(A) would then read as,
\parnoi
\emph{Assumption~(A): }$h\in \smooth{[0,R]}$, $h(r)>0$ for $r\in (0,R]$, $h(0)=0$, $h'(0)=1$ and all even order derivatives of $h$ are zero at $r=0$.
\parnoi
We begin with the following proposition describing separation of variables for Steklov eigenforms on warped products.
\begin{prop}\label{prop: sepn of variables for one bdry}
Let $(M^n,g)$ be a warped product manifold ($n\geq3$) as in Section~\ref{subsec: warped prod mfds}, with connected boundary. Let $p\in\{1,\dots,n-2\}$. Then, any Steklov eigenform associated with $\Lm^{(p)}$ is of the form $\varphi(r,x)=\psi(r)\om(x)$, where $\om$ is a coclosed eigenform on $\Snm$ (say, corresponding to an eigenvalue $\lmpmm=(m+p)(n+m-p-2)$) and $\psi(r)$ is a non-trivial solution of the ODE
\eqn{
\begin{cases}
\psi''(r)+(n-2p-1)\frac{h'(r)}{h(r)}\psi'(r)-\frac{\psi(r)}{h(r)^2}\lmpmm=0, \qquad r\in(0,R],\\
\psi(0)=0.
\end{cases}
}
The corresponding DtN eigenvalue is given by $\frac{\psi'(R)}{\psi(R)}$, and    is equal to $\karsgp_{(m)}$.
\end{prop}

In~\cite{RS14}, the eigenvalues of Raulot-Savo's DtN operator $T^{(p)}$ were computed for the Euclidean unit ball, treating it as a warped product with $h(r)=r$. Those computations apply to more general contexts as well, which we make use of in this proof. The article also has a description of eigenvalues of the Hodge Laplacian on the closed manifold $\mathbb{S}^{n-1}$~\cite[Section~1.2]{RS14}. Of the three families of eigenvalues listed there, we consider the family of eigenvalues associated with coclosed eigenforms in order to use these results for the operator $\Lm^{(p)}$. Recall that $\Lm^{(p)}$ vanishes on the space of exact forms. Thus, we are interested in the family of eigenvalues $\lmpmm\df(m+p)(n+m-p-2)$ corresponding to the eigenforms of $\lap\p_\Snm$, given by restrictions to $\mathbb{S}^{n-1}$ of homogenous $m^{th}$ degree harmonic polynomial forms in $\RR^n$ whose interior product with the radial vector field in $\RR^n$ is zero.
\par
We first use separation of variables to derive certain conditions for a $\varphi\in\Om^p(M)$ of the form $\psi(r)\om(x)$ to be a Steklov eigenform, where $\om\in\text{cc}\Om^p(\Snm)$ is an eigenform of the Hodge Laplacian on $\Snm$ and $\psi(r)$ is a smooth function on $[0,R]$. In principle, an eigenform can also have additional terms of the form $dr\wedge \eta(x)$, $\eta\in\Om^{p-1}(\Snm)$, but we restrict our attention to those of the form $\psi(r)\om(x)$. Let $\lm$ be the eigenvalue associated with $\om$, i.e.,
\begin{gather}
    \lap^{(p)}_{\mathbb{S}^{n-1}}\om=\lm\om,\label{closed eigval on bdry}\\
    \delta_{\mathbb{S}^{n-1}}\om=0.\label{coclosed on bdry}
\end{gather}
For $\varphi(r,x)$ to be a $\Lm^{(p)}$-eigenform, we require that
\begin{gather}    \lap^{(p)}_M\varphi=0,\label{harmonic}\\    \delta_M\varphi=0,\label{coclosedon M}\\
i_\normal d_M\varphi=\sg i^*\varphi,\label{steklov condn}
\end{gather}
where $\sg$ would then be the DtN eigenvalue associated with $\varphi$. Using the relations from~\cite[Section~2.3.2]{RS14} for equations~\eqref{harmonic} to~\eqref{steklov condn}, we have
\begin{gather*}
    \frac{\psi(r)}{h(r)^2}\lap_{\mathbb{S}^{n-1}}\om(x) - \Big[\psi''(r)+(n-2p-1)\frac{h'(r)}{h(r)}\psi'(r)\Big]\om(x)-2dr\wedge\psi(r)\frac{h'(r)}{h(r)^3}\delta_{\mathbb{S}^{n-1}}\om(x)=0,\\
    \frac{\psi(r)}{h(r)^2}\delta_{\mathbb{S}^{n-1}}\om(x)=0\eqtext{and}\\
    i_{\pa_r}(dr\wedge (\psi'(r)\om(x))+\psi(r)d_{\mathbb{S}^{n-1}}\om(x))|_{r=R}=\sg\,\psi(R)\om(x).
\end{gather*}
Since $\om$ satisfies equations~\eqref{closed eigval on bdry} and~\eqref{coclosed on bdry}, we get that
\begin{gather}
    \frac{\psi(r)}{h(r)^2}\lm-\Big[\psi''(r)+(n-2p-1)\frac{h'(r)}{h(r)}\psi'(r)\Big]=0 \eqtext{and}\label{ode}\\
    \sg=\frac{\psi'(R)}{\psi(R)}. \label{sg expression}
\end{gather}
\begin{remark}\label{rmk: coclosedness for test forms}
We have from~\cite[Section~2.3.2]{RS14} that for a $p$\nobreakdash-form $\psi(r)\om(x)$ on $[0,R]\times \Snm$ with the metric $dr^2\oplus h(r)^2 g_\Snm$,
$$\delta (\psi(r)\om(x))=\frac{\psi(r)}{h(r)^2}\delta_\Snm\om(x),$$
where $\delta_\Snm$ denotes the codifferential with respect to $\Snm$. Consequently, if $\om$ is coclosed with respect to $\Snm$, then $\psi(r)\om(x)$ is coclosed with respect to the warped product manifold, for any choice of warping function and any $\psi \in \smooth{[0,R]}$. We use this property in the later sections to ensure that the test $p$\nobreakdash-forms being considered are coclosed with respect to appropriate metrics. 
\end{remark}\vspace{-2ex}
The following technical lemma is needed in the proof of Proposition~\ref{prop: sepn of variables for one bdry}.
\begin{lem}\label{lem: relation bw G^p}
    If $\varphi\in \Om^p(M)$ such that $\varphi(r,x)=\psi(r)\om(x)$ for some $\om\in\Om^p(\Snm)$ and \hbox{$\psi\in\smooth{[0,R]}$}, then
    \eqn{    
    \ip{M}{d_M\varphi}{d_M\varphi}=\frac{\psi(r)^2}{h(r)^{2p+2}}\ip{\Snm}{d_\Snm\om}{d_\Snm\om}+\frac{(\psi'(r))^2}{h(r)^{2p}}\ip{\Snm}{\om}{\om},
    }
    where $\ip{M}{}{}$, $\ip{\Snm}{}{}$ denote the pointwise inner products with respect to $M$, $\Snm$, and $d_M$, $d_\Snm$ denote the exterior derivatives with respect to $M$, $\Snm$, respectively.
\end{lem}

\begin{proof}
We have
    \begin{align} \label{G p+1 expansion}
        \ip{M}{d_M\varphi}{d_M\varphi}&=\bigip{\psi(r)d_\Snm\om+\psi'(r)dr\wedge\om}{\psi(r)d_\Snm\om+\psi'(r)dr\wedge\om}_M\notag\\
        &\begin{multlined}=\psi(r)^2\ip{M}{d_\Snm\om}{d_\Snm\om}+(\psi'(r))^2 \ip{M}{dr\wedge\om}{dr\wedge\om}\notag\\ 
        +2\psi(r)\psi'(r)\ip{M}{d_\Snm\om}{dr\wedge\om}.\end{multlined} 
    \end{align}
For the warped product metric $g=dr^2\oplus h(r)^2 g_{\mathbb{S}^{n-1}}$ on $M$,
\begin{gather*}
\ip{M}{d_\Snm\om}{d_\Snm\om}=\frac{1}{h(r)^{2p+2}}\ip{\Snm}{d_\Snm\om}{d_\Snm\om},\\
\ip{M}{dr\wedge\om}{dr\wedge\om}=\frac{1}{h(r)^{2p}}\ip{\Snm}{\om}{\om}\qquad \text{and}\\
\ip{M}{d_\Snm\om}{dr\wedge\om}=0.
\end{gather*}
The lemma follows.
\end{proof}
\noi We proceed as in~\cite[Proposition~9]{Xio22} to prove Proposition~\ref{prop: sepn of variables for one bdry}.
\parnoi \textbf{Proof of Proposition~\ref{prop: sepn of variables for one bdry}.} Let $\{\om_k\}_{k\in \NN}$ be an orthonormal eigenbasis for the Hodge Laplacian on the space of coclosed forms $\text{cc}\Om^p(\Snm)$, corresponding to the eigenvalues $0<\lmp_1\leq\lmp_2\leq\dots\nearrow\infty$. In contrast to the case of smooth functions, there are no harmonic $p$\nobreakdash-forms on $\Snm$ for $1\leq p\leq n-2$. For each $k\in \NN$, let $\psi_k$ be the solution to the ODE
\eqn{
\begin{cases}
\psi_k''(r)+(n-2p-1)\frac{h'(r)}{h(r)}\psi_k'(r)-\frac{\psi_k(r)}{h(r)^2}\lmp_{(m(k))}=0, \qquad r\in(0,R],\\
\psi_k(0)=0,\quad \psi(R)=1.
\end{cases}
}
It follows that the restrictions of $\varphi_k=\psi_k\om_k$ to the boundary gives the set of all possible DtN eigenforms of $\Lm^{(p)}$ by the usual argument that they span the $L^2$-completion of the space of coclosed forms on $\Snm$. We now show that $\{\varphi_k\}_{k\in \NN}$ span the space of smooth, coclosed, harmonic $p$\nobreakdash-forms in $M$ that are orthogonal to $\mathcal{H}_D\Om^p(M)$. Firstly, we have that $\varphi_k$ are harmonic and coclosed, in view of the previous discussion while deriving ODE~\eqref{ode}. Given a coclosed, harmonic $p$\nobreakdash-form $\varphi$ on $M$ orthogonal to $\mathcal{H}_D\Om^p(M)$, its restriction to the boundary may be written as
    \eqn{
    i^*\varphi=\sum_{k\in\NN} c_k\om_k,\qquad c_k\in\RR.
    }
From the bounds on the Sobolev norms of harmonic extensions of coclosed $p$\nobreakdash-forms on $\pam$ (see~\cite[Lemma 3.4.7]{hodge_text} and~\cite[Section~3.2]{Kar19}), we have that $\{\sum_{k=1}^N c_k \varphi_k\}_{N\in\NN}$ is Cauchy in $H^2\Om^p(M)$ and the $p$\nobreakdash-form given by $\sum_{k\in\NN} c_k\varphi_k$ is well-defined, harmonic and coclosed in $M$. Further, it is equal to $i^*\varphi$ on $\pa M$. Since there exists a unique $p$\nobreakdash-form in the orthogonal complement of $\mathcal{H}_D\Om^p(M)$ satisfying these conditions, we conclude that \hbox{$\varphi=\sum_{k\in\NN} c_k\varphi_k$ in $M$}. Thus, any such $\varphi$ can be expressed as an infinite sum in terms of $\varphi_k$, including the Steklov eigenforms on $M$, in particular.
\parnoi
Consider the first eigenvalue $\karsgp_{1}$, and let $\varphi$ be an element of its eigenspace, which can be written as $\sum_{k\in\NN} c_k\varphi_k$. We know that for $i\neq j$, $\inn{L^2\Om^p(\Snm)}{\om_i}{\om_j}=0$ and using that $\delta\varphi_i=\delta\varphi_j=0$,
\begin{align*}
\inprod{M}{d\varphi_i}{d\varphi_j}&=\inprod{M}{\lap \varphi_i}{\varphi_j}+\inprod{\pam}{i_\normal d\varphi_i}{i^*\varphi_j}\tag{by Green's identity~\eqref{green's identity}}\\
&=\inprod{\pam}{i_\normal d\varphi_i}{i^*\varphi_j}\tag{since $\varphi_i$ is harmonic}\\
&=\gamma \inprod{\pam}{\om_i}{\om_j}\tag{$\gamma$ denoting the eigenvalue corresponding to $\varphi_i$}\\
&=0.
\end{align*}\nopagebreak[0]
Then we see that,
\eqn{
\karsgp_1=\frac{\sum_{k\in \NN}c_k^2\bigintsss_M \ip{}{d\varphi_k}{d\varphi_k}\,\dvm}{\sum_{k\in \NN}c_k^2}.
}
Using Lemma~\ref{lem: relation bw G^p} and Green's identity~\eqref{green's identity},
\begin{align*}
    &\int_M \ip{}{d\varphi_k}{d\varphi_k}\,\dvm= \int_0^R\Big(\int_\Snm \ip{}{d\varphi_k}{d\varphi_k}\,\dvs\Big) \,h(r)^{n-1}\dr\\
    &=\int_0^R\Big(\int_\Snm \frac{\psi_k(r)^2}{h(r)^{2p+2}}\ip{}{d_\Snm\om_k}{d_\Snm\om_k}+\frac{(\psi_k'(r))^2}{h(r)^{2p}}\ip{}{\om_k}{\om_k} \,\dvs\Big) \,h(r)^{n-1}\dr\\
    &=\int_0^R\Big(\frac{\psi_k(r)^2}{h(r)^{2p+2}}\int_\Snm \ip{}{\lap_\Snm \om_k}{\om_k}\,\dvs+ \frac{(\psi_k'(r))^2}{h(r)^{2p}}\int_\Snm \ip{}{\om_k}{\om_k} \,\dvs\Big) \,h(r)^{n-1}\dr\\
    &=\int_0^R \Big(\lmp_{(m(k))}\frac{\psi_k(r)^2}{h(r)^2}+(\psi_k'(r))^2\Big)h(r)^{n-2p-1}\,\dr\tag{as $\om_k$ are chosen to be orthonormal}.
\end{align*}
Hence,
\eqn{
\karsgp_{1}=\frac{\sum_{k\in \NN}c_k^2\int_0^R \big[\lmp_{(m(k))}\frac{\psi_k(r)^2}{h(r)^2}+(\psi_k'(r))^2\big]h(r)^{n-2p-1}\,\dr}{\sum_{k\in \NN}c_k^2}
}
Since each term in the numerator is strictly increasing in $\lmpmm$, and $\lmpmm$ is strictly increasing in $m$, any increase in $m$ will result in a value greater than $\karsgp_{1}$. This argument can be repeated for higher eigenvalues, and we conclude that the eigenspace associated with $\karsgp_{(m)}$ is the same as the eigenspace of $\lmpmm$. It is also implied that $\psi_k$ is the same for each eigenform in the eigenspace for $\karsgp_{(m)}$. Finally, the claim that $\karsg_{(m)}^{(p)}=\frac{\psi'(R)}{\psi(R)}$ follows from equation~\eqref{sg expression}.
\qed
\parnoi We now prove the main theorem of this section.
\parnoi
\textbf{Proof of Theorem~\ref{thm: Xiong-type lower bound for p-forms}.}
In view of Proposition~\ref{prop: sepn of variables for one bdry}, it will suffice to show that
\eqnum{\label{main inequality of ratios}
\frac{\psi'(R)}{\psi(R)}\geq (m+p) \frac{h'(R)}{h(R)},\eqtext{for $\psi$ corresponding to $\lmpmm$,}
}as the principal curvature of $\pam$ is given by $\frac{h'(R)}{h(R)}$.
To analyse the solution $\psi$ to ODE~\eqref{ode}, we begin by rewriting the ODE as
\eqn{
\frac{1}{h(r)^{n-2p-1}}\frac{d}{dr}\big(h(r)^{n-2p-1}\psi'(r)\big)-\frac{\lmpmm\psi(r)}{h(r)^2}=0,
}
and integrate, for $p\leq \frac{n-2}{2}$, to get
\eqn{
h(r)^{n-2p-1}\psi'(r)=\lmpmm \int_0^{r} h(t)^{n-2p-3}\psi(t)\, \dt.
}
Assuming without loss of generality that $\psi$ is strictly positive in a small neighbourhood of $r=0$, we see that $\psi'(r)>0$ and $\psi(r)>0$ on $(0,R]$. Further, we calculate $\psi'(0)$.
\begin{align*}
    \psi'(0)&=\lmpmm \lim_{r\to 0}\frac{\int_0^{r} h(t)^{n-2p-3}\psi(t)\, \dt}{h(r)^{n-2p-1}}\\
    &=\lmpmm \lim_{r\to 0}\frac{h(r)^{n-2p-3}\psi(r)}{(n-2p-1)h(r)^{n-2p-2}}\\
    &=\frac{\lmpmm}{(n-2p-1)}\lim_{r\to 0}\frac{\psi(r)}{h(r)}\\
    &=\frac{\lmpmm}{(n-2p-1)}\psi'(0),\qquad\text{(as $h'(0)=1$ by Assumption~(A)).}
\end{align*}
It follows that $\psi'(0)=0$ for $p\leq \frac{n-2}{2}$. When $p=\frac{n-1}{2}$, the ODE reads as
\eqn{
\psi''=\lmpmm\frac{\psi}{h^2}.
}So,
\eqn{
\psi''(0)=\lmpmm\lim_{r\to 0}\frac{\psi(r)}{h(r)^2}=\lmpmm\lim_{r\to 0}\frac{\psi'(r)}{2h'(r)h(r)}=\frac{\lmpmm}{2}\lim_{r\to 0}\frac{\psi'(r)}{h(r)}.
}
Since $\psi''(0)$ is finite, $\psi'(0)$ has to be $0$ because $h(0)=0$.
\parnoi
As a consequence of the hypothesis that $\ric_g\geq 0$ and $\pam$ is strictly convex, the warping function satisfies (see~\cite[Lemma~8]{Xio22})
\eqnum{\label{curvature conditions implication for warping function}
h''(r)\leq 0 \quad \text{and} \quad 0<h'(r)\leq 1, \qquad \text{for}\,\,\, r\in[0,R].
}
To prove inequality~\eqref{main inequality of ratios}, consider the function
\eqn{
q(r)\df h(r)\psi'(r)-(m+p)h'(r)\psi(r) \eqtext{on $[0,R]$}.
}
Firstly, we have $q(0)=0$, since $h(0)=0$ and $\psi(0)=0$. Next,
\begin{align*}
q'&=h'\psi'+h\psi''-(m+p)h''\psi-(m+p)h'\psi'\\
&\geq h\psi''+(1-m-p)h'\psi' \eqtext{(using $h''\leq 0$ and $\psi\geq 0$)}\\
&=:u(r).
\end{align*}
Since $\psi$ satisfies the ODE~\eqref{ode}, we get
\eqn{
u=\lmpmm\frac{\psi}{h}-\frac{\lmpmm}{m+p}h'\psi',
} where $\lmpmm=(m+p)(n+m-p-2)$.
Using $\psi'(0)=0$,
\eqn{
u(0)=\lmpmm \Big(1-\frac{1}{m+p}\Big)\psi'(0)=0.
}
Further,
\begin{align}
    \frac{u'}{\lmpmm}&=\frac{\psi'}{h}-\frac{\psi h'}{h^2}-\frac{1}{m+p}h''\psi'-\frac{1}{m+p}h'\psi'' \notag\\
    &\geq\frac{\psi'}{h}-\frac{\psi h'}{h^2}-\frac{h'}{(m+p)h}\Big(\lmpmm \frac{\psi}{h}-(n-2p-1)h'\psi'\Big)\notag\\
    &=\Big(\frac{1}{h}+(n-2p-1)\frac{(h')^2}{(m+p)h}\Big)\psi'-\big(1+\frac{\lmpmm}{m+p}\big)\frac{h'\psi}{h^2}\notag\\
    &=\Big(\frac{1}{h}+(n-2p-1)\frac{(h')^2}{(m+p)h}\Big)\frac{m+p}{\lmpmm h'}\big(\lmpmm\frac{\psi}{h}-u\big) -\big(1+\frac{\lmpmm}{m+p}\big)\frac{h'\psi}{h^2}\notag\\
    &=\frac{m+p}{h^2h'}(1-(h')^2)\psi-\Big(\frac{1}{h}+(n-2p-1)\frac{(h')^2}{(m+p)h}\Big)\frac{m+p}{\lmpmm h'}u\notag\\
    &\geq -\Bigg[\frac{1}{h}+(n-2p-1)\frac{(h')^2}{(m+p)h}\Bigg]\frac{m+p}{\lmpmm h'}u.\label{technical ineq in xiong type bounds}
\end{align}
Since $p\leq \frac{n-1}{2}$, the term in the square brackets above is nonnegative. Recall that $u(0)=0$. Therefore, if $u(a)<0$ for some $a\in(0,R]$, then there exists $b\in(0,a]$ such that $u(b)<0$ and $u'(b)<0$, contradicting the inequality~\eqref{technical ineq in xiong type bounds}. Hence, we have that $u\geq 0$ on~$[0,R]$. Consequently, $q'\geq 0$ on $[0,R]$, with $q(0)=0$, giving that $q\geq 0$ on $[0,R]$. In particular, $q(R)\geq 0$. Thus, we get the required lower bounds.
\parnoi
It is evident from~\cite[Theorem~8.1]{Kar19} that the equality is attained for the Euclidean ball of radius~$R$. Furthermore, when the equality holds, i.e., $q(R)=0$, it follows that $q\equiv 0$. Then we have
\eqn{
\frac{\psi'(r)}{\psi(r)}=(m+p)\frac{h'(r)}{h(r)},\qquad \forall\, r\in [0,R].
}
We can use this to calculate
\eqn{
\psi''= (m+p)\Big[\frac{\psi'h'}{h}+\frac{\psi}{h^2}(hh''-(h')^2)\Big].
}
Using these two equations in the ODE~\eqref{ode} and simplifying, we get that
\eqn{
(m+p)hh''+\lmpmm((h')^2-1)=0.
}
But since $h''\leq 0$ and $0<h'\leq 1$ from~\eqref{curvature conditions implication for warping function}, each of the terms in this equation must be identically~$0$. Thus, we conclude that $h(r)=r$ for all $r\in [0,R]$, i.e., the equality holds only for the Euclidean ball.
\qed
\section{Proofs of Theorems\texorpdfstring{~\ref{thm: bcg ratio upperbound for p-forms}}{ 1.4} and\texorpdfstring{~\ref{thm: bcg spectral gaps upperbound for p-forms}}{ 1.5}}\label{sec: proofs of bcg thms for p-forms}
\par
The proofs of Theorems~\ref{thm: bcg ratio upperbound for p-forms}~\&\ref{thm: bcg spectral gaps upperbound for p-forms} follow on similar lines as~\cite[Theorems~1.1,~1.2,~1.5~\&~1.7]{BCG25}. We first give a variational characterisation for the eigenvalues in the case of warped products, analogous to that given in~\hbox{\cite[Section~2]{BCG25}}.
\eqnum{\label{var char in terms of coeff functions}
\karsgpkk(M,g)=\min_{\begin{subarray}{c}\psi\in H^1([0,R])\\ \psi(R)=0\end{subarray}}\frac{\int_0^R (\psi')^2 h^{n-2p-1}+\lmpkk \psi^2 h^{n-2p-3}\,\,\dr}{\psi(0)^2h(0)^{n-2p-1}},
} where $h$ is the warping function associated with the metric $g$. Note that, in view of Remark~\ref{rmk: coclosedness for test forms}, the test functions being used here need no further restriction to ensure that the corresponding test forms are coclosed with respect to $M$.
\parnoi \textbf{Proof of Theorem~\ref{thm: bcg ratio upperbound for p-forms}.} (i) Fix $n\geq3$ and $p\in \{0,\dots,n-2\}$. Let $\psi_k\in\smooth{[0,R]}$ be the function associated with the expression describing separation of variables for the eigenforms of $\karsgpkk(M)$ (see Proposition~\ref{prop: sepn of variables for one bdry}). Using $\psi_k$ as a test function in the variational characterisation~\eqref{var char in terms of coeff functions} for $\karsgpkkp$,
\begin{align}
\karsgpkkp&\leq\frac{\int_0^R (\psi_k')^2 h^{n-2p-1}+\lmpkkp \psi_k^2 h^{n-2p-3}\,\,\dr}{\psi_k(0)^2h(0)^{n-2p-1}}\notag\\
&=\frac{\int_0^R (\psi_k')^2 h^{n-2p-1}+\lmpkk \psi_k^2 h^{n-2p-3}\,\,\dr}{\psi_k(0)^2h(0)^{n-2p-1}}+\frac{\int_0^R (\lmpkkp-\lmpkk)\, \psi_k^2 h^{n-2p-3}\,\,\dr}{\psi_k(0)^2h(0)^{n-2p-1}}\notag\\
&=\karsgpkk+\frac{\lmpkkp-\lmpkk}{\lmpkk}\frac{\int_0^R \lmpkk \psi_k^2 h^{n-2p-3}\,\,\dr}{\psi_k(0)^2h(0)^{n-2p-1}}\notag\\
&\leq \karsgpkk+\frac{\lmpkkp-\lmpkk}{\lmpkk} \karsgpkk=\frac{\lmpkkp}{\lmpkk}\karsgpkk, \label{ineq in proof of ratio bounds for p-forms}
\end{align}
which implies that
\eqn{
\frac{\karsgpkkp}{\karsgpkk}\leq \frac{\lmpkkp}{\lmpkk}.
}
If the equality were to hold, we need that
\eqn{
\frac{\int_0^R(\psi_k')^2 h^{n-2p-1}\,\,\dr}{\psi_k(0)^2h(0)^{n-2p-1}}=0,
}
which is not possible as it implies that $\psi_k'=0$ almost everywhere on $[0,R]$, contradicting the ODE~\eqref{ode}.
\parnoi
For $p\leq \frac{n-3}{2}$, this bound can be shown to be sharp by constructing a sequence of warping functions~$(h_\eps)_\eps$ which take large values on most of $[0,R]$ as $\eps\to 0$, as done in~\cite[Section~3.2]{BCG25} for the case of functions. When $n\geq4$ and $p\leq \frac{n-4}{2}$, for sufficiently small $\eps\in (0,1)$, let
\eqn{
\tilde{h}_\eps (r)\df
\begin{cases}
1,& r\in [0,\eps],\\
\eps^{-\frac{1}{2(n-2p-3)}}, & r\in[(2\eps,R-2\eps],\\
R-r, & r\in [R-\eps,R],
\end{cases}
}
and $h_\eps\in \smooth{[0,R]}$ be a function that is increasing on $[\eps,2\eps]$, decreasing on $[R-2\eps,R-\eps]$ and coinciding with $\tilde{h}_\eps$ everywhere else. When $n\geq 3$ and $p= \frac{n-3}{2}$, we instead use
\eqn{
\tilde{h}_\eps (r)\df
\begin{cases}
1,& r\in [0,\eps],\\
\eps^{-\frac{1}{2}}, & r\in[2\eps,R-2\eps],\\
R-r, & r\in [R-\eps,R].
\end{cases}
}
The rest of the proof is almost the same as the proof of Theorem~1.2 in~\cite[Section~3.2]{BCG25}, replacing $n$ with $(n-2p)$ appropriately. Note that, for $p>\frac{n-3}{2}$, the term $(n-2p-3)$ present in the exponent while constructing $\tilde{h}_\eps$ would be negative, and this does not allow the proof for the sharpness to be adapted.
\parnoi
(ii) For $n\geq 2$, it can be shown by a straightforward generalisation of the discussion in~\cite[Appendix]{CGG19} that a warped product manifold with connected boundary isometric to $\Snm$ is conformal to the Euclidean unit ball, and the conformal factor can be prescribed to be identically equal to~$1$ on the boundary. So, this will be true for the manifold $(M,\frac{1}{h(0)^2}g)$, where $(M,g)$ is the manifold under consideration. Further, for a Riemannian metric $\tilde{g}$ on $M$, a postive function $\varsigma\in\smooth{M}$ and any $\om\in\Om^p(M)$,
\eqn{
\int_{(M,\tilde{g})}\langle d\om,d\om \rangle_{\tilde{g}}\,\dv_{\hspace{-.20em}\tilde{g}}=\int_{(M,\varsigma\tilde{g})}\varsigma^{\frac{n}{2}-p-1}\langle d\om,d\om \rangle_{\varsigma\tilde{g}}\,\dv_{\hspace{-.20em}\varsigma\tilde{g}}.
}
If, further, $\varsigma\equiv 1$ on $\pam$, then the Rayleigh quotients with respect to $\tilde{g}$ and $\varsigma\tilde{g}$ are equal when $p=\frac{n-2}{2}$. We thus have that for $n\geq 2$ even and $p=\frac{n-2}{2}$, for every $k\in\NN$,
\begin{align*}
\karsgpkk(M,g)&=\frac{1}{h(0)}\karsgpkk(M,\frac{1}{h(0)^2}g)\\
&=\frac{1}{h(0)}\karsgpkk(\bn)\\
&=\frac{1}{h(0)}(k+p),
\end{align*}
where the last expression follows from~\cite[Theorem~8.1]{Kar19}.
\qed
\parnoi
\textbf{Proof of Theorem~\ref{thm: bcg spectral gaps upperbound for p-forms}.} When $n\geq 4$ and $p\leq \frac{n-4}{2}$, let the warping function satisfy $h\leq C$ for some constant $C>0$. For small enough $\eps>0$ as in the proof of~\cite[Theorems~1.6~\&~1.7]{BCG25} for the case of functions, we use the test function
\eqn{
\tilde{\psi}(r)=
\begin{cases}
1, & r\in [0,R-\eps],\\
\frac{R-r}{\eps}, & r\in [R-\eps,R]
\end{cases}
}
in the variational characterisation~\eqref{var char in terms of coeff functions}, to see that
\begin{align}\label{ineq in proof of spectral gaps for p-forms}
\karsgpkk\leq &\frac{\int_0^R (\tilde{\psi}')^2 h^{n-2p-1}+\lmpkk \tilde{\psi}^2 h^{n-2p-3}\,\,\dr}{h(0)^{n-2p-1}}\\
&\leq \frac{C^{n-2p-3}}{h(0)^{n-2p-1}}\Big(R\lmpkk +\frac{4}{3}\eps\Big),\qquad\text{using that $p\leq \frac{n-4}{2}$.}\notag
\end{align}
Combining with inequality~\eqref{ineq in proof of ratio bounds for p-forms},
\begin{align*}
\karsgpkkp&\leq\karsgpkk+\frac{\lmpkkp-\lmpkk}{\lmpkk}\karsgpkk\\
& \leq \karsgpkk+\frac{(\lmpkkp-\lmpkk)RC^{n-2p-3}}{h(0)^{n-2p-1}}+\frac{\lmpkkp-\lmpkk}{\lmpkk}\frac{4\eps C^{n-2p-3}}{3h(0)^{n-2p-1}}.
\end{align*}
We get the desired upper bound by taking the limit as $\eps\to 0$. If $n\geq 3$ is odd and $p=\frac{n-3}{2}$, we use in the variational characterisation~\eqref{var char in terms of coeff functions} the test function defined above to get $$\karsgpkk\leq \frac{R\lmpkk}{h(0)^2},$$ which in combination with inequality~\eqref{ineq in proof of ratio bounds for p-forms} gives the required upper bound.
\qed
%---------------------------------------------
\section{Proofs of Theorems\texorpdfstring{~\ref{thm: thm 1.8 of cgg lowerbound for one bdry compt for p-forms}}{ 1.6},\texorpdfstring{~\ref{thm: thm 1.11 of cgg lower bound for two bdry compts for p-forms}}{ 1.7} and\texorpdfstring{~\ref{thm: prop3.3 of cgg misc bounds for two bdry compts for p-forms}}{ 1.8}}\label{sec: proofs of cgg thms for p-forms}

\parnoi
\textbf{Proof of Theorem~\ref{thm: thm 1.8 of cgg lowerbound for one bdry compt for p-forms}.} For a hypersurface of revolution with one boundary component, which we may view as a warped product manifold as described in Section~\ref{subsec: warped prod mfds}, the warping function satisfies $|h'(r)|\leq 1$ for $r\in [0,L]$, $h(0)=1$ and $h(L) = 0$. These would also imply that $h(r)\geq 1-r$ for $r\in [0,1]$, and that $L\geq 1$. See~\cite[Section~3]{CGG19}.
\parnoi (i) Fix $n\geq 3$ and $p\leq \frac{n-3}{2}$. Recall that we have the separation of variables for the Steklov eigenforms, from Proposition~\ref{prop: sepn of variables for one bdry}. As before, let $\{\om_k\}_{k\in \NN}$ be an orthonormal eigenbasis corresponding to $\{\lmpk\}_{k\in\NN}$, the spectrum of the Hodge Laplacian acting on the space of coclosed forms on $\Snm$. If $\phi=\sum_{j\in \NN} \alpha_j(r) \om_j(p)$ is a coclosed eigen $p$\nobreakdash-form on $M$, where $\{\alpha_j\}$ are scaled versions of $\{\psi_j\}$ such that $\inprod{\pam}{i^*\phi}{i^*\phi}=1$, we observe that the Rayleigh quotients of $\phi$ with respect to $M$ (with the metric $dr^2\oplus h(r)^2 g_\Snm$ on $[0,L]\times \Snm$) and $\bn$ (with the metric $dr^2\oplus r^2 g_\Snm$ on $[0,1]\times \Snm$) satisfy
\begin{align*}
\ray_M(\phi)&=\inprod{M}{d\phi}{d\phi}\\
&=\sum_{j\in \NN}\int_0^L (\alpha_j')^2 h^{n-2p-1}+\alpha_j^2 \lmp_j h^{n-2p-3} \,\,\dr\\
&\geq \sum_{j\in \NN}\int_0^1 (\alpha_j')^2 (1-r)^{n-2p-1}+\alpha_j^2 \lmp_j (1-r)^{n-2p-3} \,\,\dr\tag{using that $p\leq \frac{n-3}{2}$}\\
&=\ray_\bn(\phi).
\end{align*}
Let $\mathscr{E}_k$ be the span of eigenforms associated with $\{\karsgp_i(M)\}_{i=1}^k$. Restricting these eigenforms, defined on \hbox{$[0,L]\times \Snm$}, to the domain \hbox{$[0,1]\times \Snm$}, we get a $k$-dimensional space of $p$\nobreakdash-forms on~$\bn$. Noting that these belong to $H^1\Om^p(\bn)$ by a similar argument as that of~\cite[Lemma~3.4]{CGG19}, and that they are coclosed with respect to $\bn$ by Remark~\ref{rmk: coclosedness for test forms}, we use them in the min-max characterisation for the Steklov eigenvalues of $\bn$ to get
\eqn{
\karsgpk(\bn) \leq \max_{0\neq\phii\in\mathscr{E}_k} \ray_\bn(\phii)\leq \max_{0\neq\phii\in\mathscr{E}_k} \ray_M(\phii)=\karsgpk(M).
}
Further, it is easy to see that the equality holds if and only if $L=1$ and the warping function corresponds to that of $\bn$, i.e., $h(r)=1-r$.
\parnoi
(ii) When $n\geq 2$ is even and $p=\frac{n-2}{2}$, it follows from Theorem~\ref{thm: bcg ratio upperbound for p-forms} that $\karsgpk(M)=\karsgpk(\bn)$ for all $k\in \NN$.
\qed
\begin{remark}\label{rmk: why not upper bds for large p}
We used the assumption that $p\leq \frac{n-3}{2}$ in the Rayleigh quotient comparison
\eqn{\ray_M(\phi)\geq \ray_\bn(\phi)} in part~(i) of the above proof. However, it is not possible to reverse this inequality for the Rayleigh quotients by instead using $p\geq \frac{n-1}{2}$ in the above steps, since we have that $L\geq 1$ for hypersurfaces of revolution with connected boundary. This restricts us from studying optimal upper bounds for $\karsgpk(M)$ when $p\geq \frac{n-1}{2}$.
\end{remark}
\par
Before proceeding to the proofs of upper and lower bounds for hypersurfaces of revolution with two boundary components, we briefly discuss separation of variables for the Steklov eigenforms of warped product manifolds with two boundary components as Proposition~\ref{prop: sepn of variables for one bdry} is not applicable in this case. Consider the space of coclosed $p$\nobreakdash-forms that are of the form $\psi(r)\om_m(p)$ where $\om_m$ is a fixed (coclosed) eigenform of the Hodge Laplacian on $\Snm$. Then the restriction of this space to~$\pam$ is a two-dimensional real vector space spanned by
\eqn{
\theta_m\df
\begin{cases}
\om_m, & \text{on } \Snm\times \{0\},\\
0, & \text{on } \Snm\times \{L\},
\end{cases}
\qquad \text{and} \qquad \tilde{\theta}_m\df
\begin{cases}
0, & \text{on } \Snm\times \{0\},\\
\om_m, & \text{on } \Snm\times \{L\}.
\end{cases}
}
One can see that this subspace is preserved by the operator $\Lm^{(p)}$. Hence, the self-adjoint operator $\Lm^{(p)}$ restricted to this finite dimensional space will have two real eigenvalues, say, $\karsgp_{m,1}$ and $\karsgp_{m,2}$. Now, since $\{\om_m\}_{m\in \NN}$ spans the space of coclosed $p$\nobreakdash-forms on $\Snm$, we have that $\bigsqcup_{m\in \NN} \{\theta_m, \tilde{\theta}_m\}$ spans the space of coclosed forms on $\Snm \sqcup \Snm$, which should also be equal to the span of all the eigenforms of $\Lm^{(p)}$. Thus, any Steklov eigenform $\phi$ must be of the form $\phi(r,x)=\psi(r)\om_m(x)$ for appropriate $\psi\in\smooth{[0,L]}$. Note that it is not possible, in general, to obtain a useful ordering of the spectrum given by the set $\bigcup_{m\in \NN}\{\karsgp_{m,1},\karsgp_{m,2}\}$, unlike the case of warped product manifolds with connected boundary considered in Proposition~\ref{prop: sepn of variables for one bdry}.
\parnoi 
\textbf{Proof of Theorem~\ref{thm: thm 1.11 of cgg lower bound for two bdry compts for p-forms}.}\vspace{1ex}\\ \nopagebreak
\emph{(i)} When $L\geq 2$: \vspace{.5ex} \\
\hspace*{1em}Let $p\leq \frac{n-3}{2}$. We modify the proof from the case of connected boundary by instead taking the restrictions of eigenforms at both ends of the manifold, i.e, to the regions corresponding to $[0,1]\times \Snm$ and $[L-1,L]\times \Snm$. We require that $L\geq 2$ for these strips to be disjoint up to a measure zero set. Further, we would have in this case that $1-r\leq h(r)\leq 1+r$ and $h(L)-r\leq h(L-r) \leq h(L)+r$ for all $r \in (0,L)$, as described in~\cite[Section~3.1]{CGG19}. Let $\{\om_k\}_{k\in \NN}$ be an orthonormal eigenbasis corresponding to $\{\lmpk\}_{k\in\NN}$, the spectrum of Hodge Laplacian acting on the space of coclosed forms on $\Snm$. If $\phi(r,x)=\sum_{j\in \NN;\,i=1,2} \alpha_{ji}(r) \om_j(x)$ is a coclosed eigen $p$\nobreakdash-form on $M$, where $\{\alpha_{ji}\}$ are scaled such that $\inprod{\pam}{i^*\phi}{i^*\phi}=1$, we observe that the Rayleigh quotients of $\phi$ with respect to $M$ (with the metric $dr^2\oplus h(r)^2 g_\Snm$) and $\bn\sqcup \bn$ (with the metric $dr^2\oplus (1-r)^2 g_\Snm$ on $[0,1]\times \Snm$ and the metric $dr^2\oplus (r-L+1)^2 g_\Snm$ on $[L-1,L]\times \Snm$) satisfy
\begin{align*}
\ray_M(\phi)&=\inprod{M}{d\phi}{d\phi}\\
&=\sum_{j\in \NN;\, i=1,2}\int_0^L (\alpha_{ji}')^2 h^{n-2p-1}+\alpha_{ji}^2 \lmp_j h^{n-2p-3} \,\,\dr\\
&\begin{multlined}
    \geq\sum_{j\in \NN;\, i=1,2}\int_0^1 (\alpha_{ji}')^2 (1-r)^{n-2p-1}+\alpha_{ji}^2 \lmp_j (1-r)^{n-2p-3} \,\,\dr\\
    \qquad+\sum_{j\in \NN;\, i=1,2}\int_{L-1}^L (\alpha_{ji}')^2 (r-L+1)^{n-2p-1}+\alpha_{ji}^2 \lmp_j (r-L+1)^{n-2p-3} \,\,\dr \tag{using $p\leq\frac{n-3}{2}$}   
\end{multlined}\\*
&=\ray_{\bn\sqcup\,\bn}(\phi).
\end{align*}
Now, using the min-max characterisation (and Remark~\ref{rmk: coclosedness for test forms}) as in the proof of Theorem~\ref{thm: thm 1.8 of cgg lowerbound for one bdry compt for p-forms} but with the test forms being the eigenforms of $M$ restricted to $\bn\sqcup\bn$ corresponding to $r\in[0,1]$ and $r\in [L-1,L]$, we get the required lower bound $\karsgpk(M)\geq \karsgpk(\bn\sqcup\bn)$. The article~\cite{CGG19} also describes the construction of a sequence of hypersurfaces of revolution to establish that the inequality is sharp in the case of functions, which can be adapted to the framework of $p$\nobreakdash-forms with straightforward modifications.
\vspace{1.5ex}\\
\emph{(ii)} When $L\leq 2$: \vspace{.5ex}\\
\hspace*{1em}\emph{(a)} Let $p\leq \frac{n-3}{2}$. With $\phi$ as before and noting that $h(r)\geq 1-\frac{L}{2}$ (see~\cite[Section~3]{CGG19}), we have
\begin{align*}
\ray_M(\phi)&=\sum_{j\in \NN;\, i=1,2}\int_0^L (\alpha_{ji}')^2 h^{n-2p-1}+\alpha_{ji}^2 \lmp_j h^{n-2p-3} \,\,\dr\\
&\geq \sum_{j\in \NN;\, i=1,2}\int_0^L (\alpha_{ji}')^2 \Big(1-\frac{L}{2}\Big)^{n-2p-1}+\alpha_{ji}^2 \lmp_j \Big(1-\frac{L}{2}\Big)^{n-2p-3} \,\,\dr \\
&\geq \Big(1-\frac{L}{2}\Big)^{n-2p-1} \sum_{j\in \NN;\, i=1,2}\int_0^L (\alpha_{ji}')^2 +\alpha_{ji}^2 \lmp_j \,\,\dr\\
&=\Big(1-\frac{L}{2}\Big)^{n-2p-1}\ray_{C_L}(\phi).
\end{align*}
The required lower bound follows from the min-max characterisation of eigenvalues (and Remark~\ref{rmk: coclosedness for test forms}) as before, using appropriate Steklov eigenforms of $M$ as test forms on $C_L$.\vspace{.5ex}\\
\hspace*{1em}\emph{(b)} Let $p\geq \frac{n-1}{2}$. Note that the exponents 
$(n-2p-1)$ and $(n-2p-3)$ associated with the warping function $h$ in the above expressions are now nonpositive, hence the inequality between the Rayleigh quotients is reversed,
\eqn{
\ray_M(\phi)\leq \Big(1-\frac{L}{2}\Big)^{n-2p-1}\ray_{C_L}(\phi).
}
To get the upper bounds, we instead use eigenforms of $C_L$ as test forms for the manifold $M$. Let $\mathscr{E}_k$ be the span of eigenforms corresponding to $\{\karsgp_i(C_L)\}_{i=1}^k$. Then the min-max characterisation (along with Remark~\ref{rmk: coclosedness for test forms}) gives
\eqn{
\karsgpk(M) \leq \max_{0\neq \phii \in \mathscr{E}_k} \ray_M(\phii) \leq \Big(1-\frac{L}{2}\Big)^{n-2p-1} \max_{0\neq \phii \in \mathscr{E}_k} \ray_{C_L}(\phii)=\Big(1-\frac{L}{2}\Big)^{n-2p-1}\karsgpk(C_L).
} \qed

\parnoi \nopagebreak
\textbf{Proof of Theorem~\ref{thm: prop3.3 of cgg misc bounds for two bdry compts for p-forms}.}\vspace{1ex}\\ \nopagebreak
\emph{(i)} When $p\leq \frac{n-3}{2}$: \vspace{.5ex}\\
\hspace*{1em}As observed in~\cite[Proposition~3.3]{CGG19}, the warping function satisfies $h(r)\leq 1+\frac{L}{2}$ for all $r\in [0,L]$. Let $\phii_l(r,x)=\sum_{j\in \NN} \sum_{i=1,2} \alpha_{ji}(r) \om_j(x)$ be an eigenform corresponding to $\karsgp_l(C_L)$, normalised by $\inprod{\pam}{i^*\phii_l}{i^*\phii_l}=1$. Then,
\begin{align*}
\ray_M(\phii_l)&=\sum_{j\in\NN;\,i=1,2}\int_0^L(\alpha_{ji}')^2 h^{n-2p-1}+\alpha_{ji}^2 \lmp_j h^{n-2p-3} \,\,\dr\\
&\leq \sum_{j\in\NN;\,i=1,2}\int_0^L(\alpha_{ji}')^2 \Big(1+\frac{L}{2}\Big)^{n-2p-1}+\alpha_{ji}^2 \lmp_j \Big(1+\frac{L}{2}\Big)^{n-2p-3} \,\,\dr \tag{we use $p\leq \frac{n-3}{2}$}\\
&\leq \Big(1+\frac{L}{2}\Big)^{n-2p-1} \sum_{j\in\NN;\,i=1,2}\int_0^L(\alpha_{ji}')^2 +\alpha_{ji}^2 \lmp_j \,\,\dr \\
&=\Big(1+\frac{L}{2}\Big)^{n-2p-1}\karsgp_l(C_L).
\end{align*}
Hence, for any $\phi\in \mathscr{E}_k\df \text{Span}(\phii_1,\dots,\phii_k)$, the following inequality holds,
\eqn{
\ray_M(\phii)\leq \Big(1+\frac{L}{2}\Big)^{n-2p-1} \karsgpk(C_L).
}
Finally, we have from the min-max characterisation, and Remark~\ref{rmk: coclosedness for test forms}, that
\eqn{
\karsgpk(M)\leq \max_{0\neq \phii\in\mathscr{E}_k}\ray_M(\phii)\leq \Big(1+\frac{L}{2}\Big)^{n-2p-1} \karsgpk(C_L).
}
\vspace{1ex}\\
\emph{(ii)} When $p\geq \frac{n-1}{2}$: \vspace{.5ex}\\
\hspace*{1em}This time, we use appropriate Steklov eigenforms of $M$ as test forms on $C_L$. Let $\phii_l(r,x)=\sum_{j\in \NN;\,i=1,2} \alpha_{ji}(r) \om_j(x)$ be an eigenform corresponding to $\karsgp_l(M)$, normalised as usual by
\eqn{\inprod{\pam}{i^*\phii_l}{i^*\phii_l}=1.}
The $p$\nobreakdash-forms $\varphi_l$, defined on $[0,L]\times \Snm$, may also be viewed as $p$\nobreakdash-forms on the standard cylinder $C_L$ with the metric $dr^2\oplus g_\Snm$, with the associated Rayleigh quotients $\ray_{C_L}(\varphi_l)$. Then,
\begin{align*}
\Big(1+\frac{L}{2}\Big)^{n-2p-1} \ray_{C_L}(\phii_l)&=\Big(1+\frac{L}{2}\Big)^{n-2p-1}\sum_{j\in\NN;\,i=1,2}\int_0^L(\alpha_{ji}')^2 +\alpha_{ji}^2 \lmp_j \,\,\dr \\
&\leq \sum_{j\in\NN;\,i=1,2}\int_0^L(\alpha_{ji}')^2 \Big(1+\frac{L}{2}\Big)^{n-2p-1}+\alpha_{ji}^2 \lmp_j \Big(1+\frac{L}{2}\Big)^{n-2p-3} \,\,\dr\\
&\leq \sum_{j\in\NN;\,i=1,2}\int_0^L(\alpha_{ji}')^2 h^{n-2p-1}+\alpha_{ji}^2 \lmp_j h^{n-2p-3} \,\,\dr \tag{we use $p\geq \frac{n-1}{2}$}\\
&=\karsgp_l(M).
\end{align*}
Hence, for any $\phi\in \mathpzc{E}_k\df \text{Span}(\phii_1,\dots,\phii_k)$, the following inequality holds,
\eqn{
\Big(1+\frac{L}{2}\Big)^{n-2p-1} \ray_{C_L}(\phii)\leq \karsgpk(M).
}
We again have from the min-max characterisation, and Remark~\ref{rmk: coclosedness for test forms}, as usual, that
\eqn{
\Big(1+\frac{L}{2}\Big)^{n-2p-1} \karsgpk(C_L) \leq \Big(1+\frac{L}{2}\Big)^{n-2p-1} \max_{0\neq \phii\in\mathpzc{E}_k}\ray_{C_L}(\phii)\leq \karsgpk(M).
}
Analogous to~\cite[Lemma~2]{CGG19}, it is possible to express the Steklov eigenvalues of $C_L$ for $p$\nobreakdash-forms in terms of the coclosed eigenvalues of the Hodge Laplacian on the boundary and the meridian length $L$. We do not know the order in general but the spectrum is given by the set \eqn{\bigcup_{k\in\NN}\left\{\sqrt{\lmpk(\Snm)}\tanh\Big(\sqrt{\lmpk(\Snm)}\frac{L}{2}\Big), \sqrt{\lmpk(\Snm)}\coth\Big(\sqrt{\lmpk(\Snm)}\frac{L}{2}\Big)\right\}.} Thus, we see that $\karsgpk(C_L)\to 0$ as $L\to 0$.
\qed
%---------------------------------------------
\section*{Acknowledgement}
The author is a PhD student at the University of Bristol and is grateful to his supervisor Asma~Hassannezhad for her valuable suggestions. He also thanks Bruno Colbois and Katie Gittins for their helpful feedback on the preprint. The author acknowledges the support from the EPSRC grant~EP/T030577/1 which facilitated a visit to Bristol shortly before the commencement of his doctoral studies.
%---------------------------------------------

%---------------------------------------------

\end{document}